\documentclass[11pt]{article}
\usepackage{amsmath}          
\usepackage{amssymb}  
\usepackage{amsmath,amsthm} 
\usepackage{enumerate}
\usepackage{stmaryrd}
\usepackage{url} 
\usepackage{mathrsfs}       % provides script font by \mathscr

\usepackage[all]{xy}

\usepackage{scalerel}   % for \bo and \di

 \DeclareMathOperator{\image}{Im}

\newcommand{\bb}[1]{\mathbb{#1}}
   \newcommand{\A}{\bb{A}}
   \newcommand{\B}{\bb{B}}
   \newcommand{\LL}{\bb{L}}  
   \newcommand{\M}{\bb{M}}    
   \newcommand{\C}{\bb{C}} 
   
 \newcommand{\f}{\mathbf{f}}  
   
\newcommand{\cS}{\mathcal{S}}

\newcommand{\CC}{\mathscr{C}}
\newcommand{\sL}{\mathscr{L}}
\newcommand{\V}{\mathscr{V}}

\newcommand{\al}{\alpha}

\newcommand{\lam}{\lambda}
\newcommand{\Lam}{\Lambda}
\newcommand{\om}{\omega}
\newcommand{\Om}{\Omega}
\newcommand{\ph}{\varphi}
\newcommand{\sg}{\sigma}   
\newcommand{\Sig}{\Sigma}   
\newcommand{\thet}{\theta}     
\newcommand{\Ups}{\Upsilon}    

\newcommand{\br}[1]{[\![#1]\!]}  

\def\bo{\mathop{\scalerel*{\Box}{X}}\kern-1pt}
\def\di{\mathord{\scalerel*{\Diamond}{gX}}}

\newcommand{\epi}{\twoheadrightarrow}

\newcommand{\join}{\bigvee}
\renewcommand{\le}{\leqslant}  
\newcommand{\lo}{^{\triangledown}}  
\newcommand{\meet}{\bigwedge}
\newcommand{\mono}{\rightarrowtail}
\newcommand{\ov}[1]{\overline{#1}}
        
\newcommand{\sub}{\subseteq}

\newcommand{\up}{^{\vartriangle}}

\newcommand{\vc}[2]{#1_0,\ldots{},#1_{#2-1}}
  
\newcommand{\<}{\langle}  
\renewcommand{\>}{\rangle}  

\makeatletter
\def\blfootnote{\xdef\@thefnmark{}\@footnotetext}
\makeatother
\numberwithin{equation}{section}         %%Equations numbered within sections

\theoremstyle{plain}
\newtheorem{theorem}{Theorem}[section]
\newtheorem{lemma}[theorem]{Lemma}

\theoremstyle{definition}

\newtheorem{example}[theorem]{Example}

\begin{document}

%%%%%%%%%%%%%%%%%%%%%%%%%%%%%%%%%%%%%%%%%%%%%%%%%%%%%%%%%%%%%%%%%%%
%%  FRONT MATTER                                                 %%
%%%%%%%%%%%%%%%%%%%%%%%%%%%%%%%%%%%%%%%%%%%%%%%%%%%%%%%%%%%%%%%%%%%

\title{Definable operators on stable set lattices}
\author{Robert Goldblatt
\\
Victoria University of Wellington}
\date{}    %% submitted to arXiv
\maketitle


\begin{abstract}
A fundamental result from Boolean modal logic states that a first-order definable class of Kripke frames defines a logic that is validated by all of its canonical frames. We generalise this to the level of non-distributive logics that have a relational semantics provided by structures based on polarities. Such structures have  associated complete lattices of stable subsets, and these have been used to construct canonical extensions of lattice-based algebras. We study classes of structures that are closed under ultraproducts and whose stable set lattices have additional operators that are first-order definable in the underlying structure. We show that such classes generate varieties of algebras that are closed under canonical extensions. The proof makes use of a relationship between canonical extensions and MacNeille completions.
\end{abstract}

%\maketitle            

%%%%%%%%%%%%%%%%%%%%%%%%%%%%%%%%%%%%%%%%%%%%%%%%%%%%%%%%%%%%%%%%%%%
%%  MAIN MATTER                                                  %%
%%%%%%%%%%%%%%%%%%%%%%%%%%%%%%%%%%%%%%%%%%%%%%%%%%%%%%%%%%%%%%%%%%%

\section{Introduction}\label{secintro}

A significant theorem of Fine \cite{fine:conn75} states that any normal modal logic that is characterised by a first-order definable class of Kripke frames must be valid in all its canonical frames. This was an important contribution to
clarifying the relationship between first-order logic and modal logic under Kripke semantics. The theorem was
generalised in \cite{gold:vari89} to  a result about the closure under canonical extensions of certain varieties (i.e.\ equationally definable classes) of Boolean algebras with operators. Here we generalise it further to varieties of non-distributive lattices with operators.
In so doing we preserve a core feature of Kripke semantics, namely that it interprets propositional formulas as first-order definable operations on subsets of a relational structure.

We work with the notion of a \emph{polarity}  $P=(X,Y,R)$ as consisting of a binary relation $R$ from a set $X$ to a set $Y$. In the same sort of way that Kripke frames have been used to model  Boolean modal logics,  polarities have been used to provide a relational semantics for various non-distributive substructural logics.
These include the implication-fusion fragments of relevant logic, BCK logic and others \cite{dunn:cano05,gehr:gene06}; the full Lambek-Grishin calculus \cite{cher:gene12} and linear logic \cite{coum:rela14}; and logics with unary modalities
 \cite{conr:cate16,conr:algo16}. Algebraically these systems are modelled by (typically non-distributive) \emph{lattice expansions}, i.e.\ lattices with additional operations.
 
Just as a Kripke frame has a modal algebra of all subsets of the frame, each polarity $P$ has an associated complete lattice $P^+$ whose members are certain \emph{stable} subsets of $X$. We call $P^+$ the \emph{stable set lattice} of $P$. In the converse direction, any lattice can be embedded into the stable set lattice of some polarity. That construction has been applied by Gehrke and Harding \cite{gehr:boun01} to develop a notion of {canonical extension} of any  lattice-based algebra. 

Canonical extensions were first introduced by J{\'o}nsson and Tarski \cite{jons:bool51} for Boolean algebras with \emph{operators} (join-preserving operations), and they play a significant role in the meta-theory of modal logics. They are closely connected with the notion of canonical frame: the  algebra of all subsets of a canonical frame of a modal logic is the canonical extension of its associated Lindenbaum algebra.
 Canonical extensions are involved in explaining the structural relationships underlying Fine's theorem.  The property of a logic being validated by its canonical frames was generalised in \cite{gold:vari89} to that of a variety of Boolean algebras with operators being closed under canonical extensions. Then the theorem from \cite{fine:conn75} was generalised to the result that if a class $\cS$ of relational structures is closed under ultraproducts, then the subset algebras of the members of $\cS$ generate a variety of Boolean algebras with operators
 that is closed under canonical extensions (see \cite{gold:fine20} for more on the background and significance of this theorem).
 
 The present paper continues a programme of lifting results like this from the modal setting to the context of polarities and lattice-based algebras. It follows on from \cite{gold:cano18}, where the concept of a \emph{canonicity framework} was introduced as an axiomatic formulation of a set of relationships between a class $\CC$ of abstract algebras and a class 
 $\Sig$ of ``structures''. It was shown that within any canonicity framework, the generalisation of Fine's theorem can be derived: each ultraproducts-closed subclass of $\Sig$ gives rise to a subvariety of $\CC$ that is closed under canonical extensions. Moreover it was shown that the axioms of a canonicity framework are fulfilled by taking $\CC$ to be the class of bounded lattices and $\Sig$ to be the class of all polarities.
 
 Here we will extend this analysis by building canonicity frameworks for which $\CC$ is a variety of lattice expansions whose additional operations are operators (join preserving) or dual operators (meet preserving). The key  idea is that of a \emph{first-order definable} operation on a stable set lattice, an idea that goes to the heart of Kripke's  semantical interpretation of the modalities $\bo$ and $\di$. On the algebra of subsets of a Kripke frame $(X,R)$, the modal connectives can be interpreted as  operations assigning to each set $A\sub X$ the sets
 $$
 \bo A = \{x: \forall y (xRy\to y\in A)\}\enspace \text{and}\enspace \di A = \{x: \exists y (xRy\ \& \ y\in A) \}.
$$
The expressions defining the members of these sets can be seen as first-order formulas in the binary predicate $xRy$ and the unary predicate $y\in A$,  leading to the `standard translation' of the propositional modal language into a first-order language \cite[\S 2.4]{blac:moda01}. This ability to relate modal logic to a fragment of first-order logic does much to account for the success of the relational semantics revolution.

We will give a formal account of what it is for a stable set lattice to be closed under an operation defined by a first-order formula. Then for a given class $\Sig$ of similar polarity-based structures and a set $\Phi$ of formulas,  $\Sig_\Phi$ is defined to be the class of those members of  $\Sig$ whose stable set lattices are closed under the operations defined by the members of $\Phi$. When these definable operations are completely join preserving or completely meet preserving, a canonicity framework can be constructed with $\Sig_\Phi$ as one of its ingredients. The outcome is that the generalisation of Fine's theorem holds  for all varieties of lattice-based algebras generated by ultraproducts-closed subclasses of $\Sig_\Phi$.

Verifying the framework axioms involves establishing  properties of ultraproducts of stable set lattices and of polarity structures. A critical property is that the canonical extension  $(P^+)^\sg$ of the stable set lattice $P^+$ of a polarity structure $P$ can be embedded into the stable set lattice $(P^U)^+$ of some ultrapower $P^U$ of $P$. We prove this by showing that 
  $(P^U)^+$ is a MacNeille completion of $(P^+)^\sg$ and  invoking a result of \cite{gehr:macn06} about the relationship between  MacNeille completions and canonical extensions. 
  
In the next section we review some basic theory about these two kinds of completion of a lattice expansion. In Section 3 we discuss polarities and their stable set lattices, and develop the notion of a definable operation on such a lattice, providing examples of this notion involving residuated lattices and modal operators. Section 4 is about ultraproducts of polarities and stable lattices, and proves the fundamental properties that are needed in Section 5, where we establish the existence of canonicity frameworks and obtain our main theorem generalising Fine's canonicity theorem to lattice expansions.

\section{Complete extensions}\label{sec1}  

This section reviews the  notions of canonical extension and MacNeille completion of a lattice-based algebra.   
We deal throughout the paper with bounded lattices, and view them as algebras of the form $(\LL,\land,\lor,0,1)$, with binary operations of meet $\land$ and join $\lor$, least element 0 and greatest element 1.  The partial order of a lattice is denoted  $\le$, and the symbols $\join$ and $\meet$ are used for the join and meet of a set of elements, when these exist. Lattice homomorphisms are assumed to preserve 0 and 1. A surjective homomorphism (\emph{epi}morphism) may be denoted by $\epi$, and an injective one (\emph{mono}morphism) by  $\mono$. The notation $f[S]$ will be used for the image $\{fa:a\in S\}$ of a set $S$ under function $f$.

A function $\thet\colon\LL\to\M$ between lattices is called \emph{isotone} if it is order preserving: $a\le b$ implies $\thet a\le \thet b$. It is   \emph{antitone} if it is order reversing: $a\le b$ implies $\thet b\le\thet a$. It is a \emph{lattice embedding}  if it is a monomorphism of bounded lattices. A lattice embedding is always an \emph{order embedding}, i.e.\ has 
$a\le b$ iff $\thet a\le \thet b$.  A function of one or more coordinates is called \emph{monotone} if in each coordinate it is isotone or antitone.

A finitary operation $f\colon \LL^n\to\LL$ on a lattice is an \emph{operator} if it preserves binary joins in each coordinate. A \emph{normal operator} preserves the least element in each coordinate as well, hence preserves all  finite joins in each coordinate, including the empty join 0. A \emph{complete operator} preserves all existing non-empty joins in each coordinate, while a \emph{complete normal operator} preserves the empty join as well. By iterating the join preservation in each coordinate one can show that if $f$ is a  complete normal operator, then
\begin{equation}\textstyle   \label{joincomplete}
f(\join A_0,\dots,\join A_{n-1})=\join\{f(\vc{a}{n}):a_i\in A_i \text{ for all }i<n\}.
\end{equation}
A \emph{dual operator} (\emph{normal} dual operator, \emph{complete} dual operator, \emph{complete normal} dual operator) is a finitary operation that preserves binary meets (finite meets, non-empty meets, all meets) in each coordinate. Preservation of the empty meet means preservation of the greatest element $1$.

A \emph{completion} of lattice $\LL$ is a pair $(\thet,\C)$ with $\C$ a complete lattice and $\thet\colon\LL\mono\C$  a lattice embedding.
An element of $\C$ is called \emph{closed} if it is a meet of elements from the image $\thet[\LL]$ of $\LL$, and \emph{open} if it is a join of elements from  $\thet[\LL]$. The  set of closed elements of the completion is denoted $K(\C)$, 
and the set of open elements is denoted $O(\C)$.

A completion $(\thet,\C)$ of $\LL$ is \emph{dense} if $K(\C)$ is join-dense and $O(\C)$ is meet-dense in $\C$, i.e.\ if  every member of $\C$ is both a join of closed elements and a meet of open elements. 
A completion is \emph{compact} if for any set $S$ of closed elements and any set $T$ of open elements such that $\meet S\le\join T$,  there are finite sets $S'\sub S$ and $T'\sub T$ with $\meet S'\le\join T'$.

A \emph{canonical extension} of bounded lattice $\LL$ is a completion $(\thet,\LL^\sg)$  of $\LL$ which is dense and compact. It is shown in \cite{gehr:boun01} that a dense and compact completion exists for any $\LL$, and that any two such completions are isomorphic by a unique isomorphism commuting with the embeddings of $\LL$.  This justifies talk of ``the'' canonical extension.

A function $f\colon\LL\to \M$ between lattices can be lifted it to a function $\LL^\sg\to\M^\sg$ between their canonical extensions in two ways, using the embeddings $\thet_\LL\colon \LL\mono\LL^\sg$ and  $\thet_\M\colon \M\mono\M^\sg$ to form the \emph{lower} canonical extension $f\lo$ and \emph{upper} canonical extension $f\up$ of $f$: see 
\cite[Definition 4.1]{gehr:boun01} where these functions are denoted $f^\sg$ and $f^\pi$ respectively. For \emph{isotone} $f$ they can be defined  for all  $x\in\LL^\sg$ as follows \cite[Lemma 4.3]{gehr:boun01}:
\begin{align*}
f\lo x &=\join\{ \meet \{\thet_\M(fa) :a\in\LL\text{ and }p\leq \thet_\LL(a)\}:x\geq p\in K(\LL^\sg) \},
\\
f\up x &=\meet\{ \join \{\thet_\M(fa) :a\in\LL\text{ and }q\geq \thet_\LL(a)\} :x\leq q\in O(\LL^\sg)\}.
\end{align*}
The maps $f\lo$ and $f\up$  have $f\lo x\leq f\up x$. They both extend $f$ in the sense that the diagram
$$
\newdir{ >}{{}*!/-8pt/@{>}}
\xymatrix{
\LL  \ar@{ >->}[d]_{\thet_\LL} \ar[r]^f  &\M \ar@{ >->}[d]^{\thet_\M} 
\\
{\LL^\sg} \ar[r]^{g} &{\M^\sg}   }
$$
commutes when $g= f\lo$ or $g= f\up$.

If $f\colon\LL^n\to\LL$ is an $n$-ary operation on $\LL$, then $f\lo$ and $f\up$ are maps from $(\LL^n)^\sg$ to $\LL^\sg$. But $(\LL^n)^\sg$ can be identified with $(\LL^\sg)^n$, since the natural embedding $\LL^n\to(\LL^\sg)^n$ is dense and compact, so this allows $f\lo$ and $f\up$ to be regarded as an $n$-ary operations on $\LL^\sg$. If $f$ is a (normal) operator, then $f\lo$ is a complete (normal) operator, and if $f$ is a (normal) dual operator, then $f\up$ is a complete (normal) dual operator \cite[Section 4]{gehr:boun01}.

A \emph{lattice expansion} (or \emph{lattice-based algebra}) is an algebra of the form
$$
\LL=(\LL_0,\{\f^\LL:\f\in\Om\}),
$$
where $\LL_0$ is a bounded lattice, $\Om$ is a set of finitary operation symbols with given arities,
and for $n$-ary $\f$, $\f^\LL$ is an $n$-ary operation on  $\LL_0$. We also call such an $\LL$ an \emph{$\Om$-lattice}.
We will take $\Om$ to be presented as the union $\Lam\cup\Ups$ of disjoint subsets $\Lam$ and $\Ups$  (`lower' and `upper' symbols, respectively). For any $\Om$-lattice $\LL$, 
define a canonical extension for $\LL$ by putting
\begin{equation}  \label{Lsigom}
\LL^\sg = (\LL_0^\sg, \{(\f^\LL)\lo:\f\in\Lam\}\cup\{(\f^\LL)\up:\f\in\Ups\}).   
\end{equation}

A \emph{MacNeille completion} of a lattice $\LL$ is a completion $\thet\colon \LL\mono\ov\LL$ of $\LL$ such that $\thet[\LL]$ is both meet-dense and join-dense in the complete lattice $\ov\LL$, i.e.\  every member of $\ov\LL$ is both a meet  of  elements of $\thet[\LL]$ and a join of  elements of $\thet[\LL]$. Every lattice has a MacNeille completion, and any two such completions are isomorphic by a unique isomorphism commuting with the embeddings of $\LL$ (see e.g.\ \cite{dave:intr90}).

For an isotone function $f\colon\LL\to\M$ between lattices,  define two functions $\ov f,\widehat{f}\colon \ov\LL\to\ov\M$ by working with the embeddings $\thet_\LL\colon \LL\mono\ov\LL$ and  $\thet_\M\colon \M\mono\ov\M$ to put
\begin{equation}     
\begin{split}
\ov f(x) &= \join\{ \thet_\M(f(a)):a\in\LL\ \&\ \thet_\LL(a)\le x\},    \label{lowerMc}
\\
\widehat f(x) &= \meet\{ \thet_\M(f(a)):a\in\LL\ \&\ x\le \thet_\LL(a)\}.   %\notag.
\end{split}
\end{equation}
$\ov f$ and $\widehat f$ are the \emph{lower} and \emph{upper} MacNeille extensions of $f$, respectively. 

An isotone $n$-ary  $f\colon \LL^n\to\LL$ thus has two extensions to $\ov{\LL^n}$, and the latter can be identified with
$(\,\ov\LL\,)^n$ because the embedding $\LL^n\mono (\,\ov\LL\,)^n$ is a MacNeille completion of $\LL^n$ \cite[Prop.~2.5]{theu:macn07}. So 
$\ov f$ and $\widehat f$ can be regarded as $n$-ary operations on $\ov\LL$. Thus we can define a MacNeille completion $\ov\LL$ of an $\Om$-lattice $\LL$ by putting
\begin{equation}  \label{Lovom}
\ov\LL = \big(\ov{\LL_0}, \big\{\ov{\f^\LL}:\f\in\Lam\big\}\cup\big\{\widehat{\f^\LL}:\f\in\Ups\big\}\big).
\end{equation}

A significant relationship between canonical extensions and MacNeille completions was established in \cite{gehr:macn06} for lattice expansions that are monotone. Whereas Fine \cite{fine:conn75} proved that a sufficiently saturated
model of a modal logic could be mapped onto a canonical frame for the logic,  \cite{gehr:macn06} worked dually with saturated extensions of algebras, showing that any monotone lattice expansion $\LL$ has an extension $\LL^*$ such that the canonical extension $\LL^\sg$ of $\LL$ is embeddable into any MacNeille completion $\ov{\LL^*}$ of $\LL^*$ by an \emph{$\Om$-monomorphism}, i.e.\ a lattice monomorphism preserving the operations indexed by $\Om$. (In fact the constructed embedding also preserves all existing joins and meets.) An extension $\LL^*$ having the required saturation can be obtained as an ultrapower $\LL^U$ of $\LL$ modulo some ultrafilter $U$, using the theory of saturation of ultrapowers \cite[\S 6.1]{chan:mode73}. Thus  \cite[Theorem 3.5]{gehr:macn06} yields the following fact.

\begin{theorem}  \label{LLsigmaembed}
For any monotone\/ $\Om$-lattice $\LL$ there exists an ultrafilter $U$ and an\/ $\Om$-monomorphism 
$\LL^\sg\rightarrowtail \ov{\LL^U}$
from the canonical extension of\/ $\LL$ into the MacNeille completion  of the ultrapower $\LL^U$.
\qed
\end{theorem}

\section{Definable operations over polarities}

A polarity $P=(X,Y,R)$ has $R\sub X\times Y$. The relation $R$ induces functions $\rho_R\colon\wp X\to\wp Y$ and $\lam_R\colon\wp Y\to\wp X$, where $\wp$ denotes powerset. Each
 set $A\sub X$ has the `right set'  $\rho_R A=\{y\in Y: \forall x\in A,xRy\}$, while each $B\sub Y$ has the `left set'
 $\lam_R B=\{x\in X: \forall y\in B,xRy\}$. The  functions $\rho_R$ and $\lam_R$ are inclusion-reversing and satisfy $A\sub\lam_R\rho_R A$ and $B\sub\rho_R\lam_R B$, i.e.\ they give a Galois connection between the posets $(\wp X,\sub)$ and $(\wp Y,\sub)$.
 A set $A\sub X$ is \emph{stable} if $\lam_R\rho_R A\sub A$ and hence $\lam_R\rho_R A=A$. A set $B\sub Y$ is \emph{stable} if $B=\rho_R\lam_R B$.
 Since in fact every $B\sub Y$ has $\lam_R\rho_R\lam_R B=\lam_R B$, the stable subsets of $X$ are precisely the sets $\lam_R B$ for all $B\sub Y$.  
 
$P^+$ is the set of all stable subsets of $X$ in $P$, ordered by set inclusion. It forms a complete bounded lattice in which 
$\meet G=\bigcap G$,
$\join G=\lam_R\rho_R\bigcup G$, $1=X$ and $0=\lam_R\rho_R \emptyset=\lam_R Y$. We call $P^+$ the \emph{stable set lattice} of $P$.
This construction  was used in \cite{gehr:boun01} to obtain a canonical extension of any lattice $\LL$  as the stable set lattice of the polarity for which $X$ is the set of filters of $\LL$, $Y$ is the set of ideals, and $xRy$ iff $x\cap y\ne\emptyset$. The embedding $\thet$ in this case has $\thet(a)=\{x\in X:a\in x\}$.

If $X=Y$ and $R$ is irreflexive and transitive, then the functions $\rho_R$ and $\lam_R$ are identical and provide an orthocomplementation making $P^+$ into an ortholattice \cite[Section 32]{birk:latt40}. In particular, if $R$ is the non-identity relation $\{(x,y):x\ne y\}$ on $X$, then  $\rho_R A=\lam_R A=$ the set complement $X-A$, all subsets are stable, and  $P^+$ is the  Boolean powerset algebra on $X$.

 We view any  polarity $P$ as a two-sorted structure for the first-order language of the signature $\sL=\{\ov X,\ov Y,\ov R\}$. Here $\ov X$ and $\ov Y$ are unary relation symbols interpreted as the sorts $X$ and $Y$ of $P$, while $\ov R$ is binary and interpreted as the relation $R$. We write $\sL$-formulas using a set $\{v_n:n<\omega\}$ of individual variables ranging over $X\cup Y$. For instance, any polarity is a model of the sentences
\begin{equation}  \label{defpolar}
\forall v_0(\ov X(v_0)\lor\ov Y(v_0)), \qquad \forall v_0\forall v_1(v_0\ov Rv_1\to\ov X(v_0)\land\ov Y(v_1)).
 \end{equation}
 
 We will form expansions of $\sL$  by adding various relation symbols denoting finitary relations on $X\cup Y$.
 For an illustration of first-order expressibility, consider a unary symbol $S$, typically interpreted as a subset of $X$. Define $\rho S(v_1)$ to be the formula $\forall v_0(S(v_0)\to v_0\ov Rv_1)$, and let $\lam\rho S(v_2)$ be $\forall v_1(\rho S(v_1)\to v_2\ov Rv_1)$. If 
  $\sL'=\sL\cup\{S\}$ and a polarity $P$ is expanded to an $\sL'$-structure $P'$ by interpreting $S$ as the set $A\sub X$, then the formula $\rho S$ defines $\rho_RA$ in $P'$, i.e.\ 
 $P'\models (\rho S)[y]$ iff $y\in\rho_RA$. Hence $\lam\rho S$ defines $\lam_R\rho_RA$.
 Thus if $\mathsf{stable}$-$S$ is the sentence $\forall v_2(\lam\rho S(v_2)\to S(v_2))$, then $\mathsf{stable}$-$S$ expresses stability of $A$, i.e.\ $P'\models\mathsf{stable}$-$S$ iff $A$ is stable.
 More generally, by replacing $S$ by any formula $\ph$ with a single free variable we can define a sentence 
 $\mathsf{stable}$-$\ph$ that is true in $P'$ iff the subset $\{x:P'\models\ph[x]\}$ of $X$ defined by $\ph$ is stable.

To develop a  notion of definable function over a polarity-based structure, fix some expansion $\sL^*$ of the signature $\sL$ for polarities.
 Let $\sL_\om^*=\sL^*\cup\{S_m:m<\om\}$, where each $S_m$ is a unary relation symbol, and for each $n<\om$ let
 $\sL_n^*=\sL^*\cup\{S_m:m<n\}$.
 If $\ph$ is a first-order $\sL_n^*$-formula with one free variable, then for each $\sL^*$-structure $P$, the formula $\ph$ defines an $n$-ary function $f^P_\ph$ on subsets of $X$ by putting, for any $A_0,\dots,A_{n-1}\sub X$, 
\begin{equation}\label{Fdef}
f^P_\ph(A_0,\dots,A_{n-1})=\{x\in X:\<P,A_0,\dots,A_{n-1}\>\models \ph[x]\},
\end{equation}
which is the subset of $X$ defined by $\ph$ in  the $\sL_n^*$-expansion  $\<P,A_0,\dots,A_{n-1}\>$ of $P$ in which each $S_m$ is interpreted as $A_m$.  

Now fix a set $\Om$ of operation symbols. 
Let $\Phi=\{\ph_\f:\f\in\Om\}$ be a set of $\sL^*_\om$-formulas indexed by $\Om$, with each $\ph_\f$ having one free variable.  Then for any class $\Sig$ of $\sL^*$-structures we define $\Sig_\Phi$ to be the class of all those $P\in\Sig$ for which the lattice $P^+$ is closed under the function $f^P_\ph$ for all $\ph\in\Phi$. For such $P$ we define the $\Om$-lattice
\begin{equation}      \label{Pptau}
P^+_\Om=(P^+,\{\f^{P^+}:\f\in\Om\}),
\end{equation}
where $\f^{P^+}$ is the restriction of the function $f^P_{\ph_\f}$ to $P^+$.

\begin{example} \textbf{Residuated lattices.}\label{Lambek}  
To describe some of the structures appearing in 
\cite{dunn:cano05,gehr:gene06,cher:gene12,coum:rela14},
let $\sL^*=\sL\cup \{ \ov T\}$ with $\ov T$ a ternary relation symbol. An $\sL^*$-structure has the form
$P=(X,Y,R,T)$. We want to have $R\sub X\times Y$ and $T\sub X\times X\times Y$,  two properties that are expressible by first-order $\sL^*$-sentences (see \eqref{defpolar}).
%$\forall v_0\forall v_1\forall v_2(\ov T(v_0\,v_1,v_2)\to\ov X(v_0)\land\ov X(v_1)\land\ov Y(v_2))$.
$P$ is then called \emph{separating} if it satisfies
\begin{align*}
&\forall x,x'\in X(\rho_R\{x\}=\rho_R\{x'\} \text{ implies } x=x'),
\\
&\forall y,y'\in Y(\lam_R\{y\}=\lam_R\{y'\} \text{ implies } y=y').
\end{align*}
$P$ is \emph{reduced} if
\begin{align*}
&\forall x\in X\exists y\in Y (\text{not }xR y \text{ and }\forall x'\in X(\rho_R\{x\}\subset \rho_R\{x'\}  \text{ implies }x'Ry)),
\\
&\forall y\in Y\exists x\in X (\text{not }xR y \text{ and }\forall y'\in Y(\lam_R\{y'\}\subset \lam_R\{y\}  \text{ implies }xRy')).
\end{align*}
$P$ is a \emph{Lambek frame} if it is separating and reduced and for all $x_0,x_1\in X$ and $y\in Y$, the following sets, which are \emph{sections of} $T$, are stable:
\begin{align*}
T[x_0,x_1,-] &=\{y'\in Y:T(x_0,x_1,y')\},
\\
T[x_0,-,y] &=\{x\in X:T(x_0,x,y)\},
\\
T[-,x_1,y]  &=\{x\in X:T(x,x_1,y)\}.
\end{align*}
These conditions defining a Lambek frame are readily expressible as first-order $\sL^*$-sentences, i.e.\  the  Lambek frames form an \emph{elementary class}.

A `fusion' operation $\otimes$ on subsets of $X$ in a Lambek frame is given by
$$
A_0\otimes A_1=\bigcap\{\lam_R\{y\}: \forall x_0,x_1\in X(x_0\in A_0\ \&\ x_1\in A_1\text{ implies }T(x_0,x_1,y)\}.
$$
$A_0\otimes A_1$ is stable, being an intersection of stable sets. Hence $P^+$ is closed under the operation 
$\otimes$.
$A_0\otimes A_1$ is not first-order definable by an $\sL^*$-formula, but it is ``first-order relative to $A_0$ and $A_1$''. If $\sL^*_2=\sL^*\cup\{S_0,S_1\}$, where $S_0$ and $S_1$ are unary relation symbols interpreted as $A_0$ and $A_1$, then 
$A_0\otimes A_1$ is defined in the  $\sL^*_2$-expansion $\<P,A_0,A_{1}\>$ of $P$, by an $\sL^*_2$-formula $\ph(v)$, namely
$$
\forall v_2\big[\,\ov Y(v_2)\land\forall v_0\forall v_1\big(S_0(v_0)\land S_1(v_1)\to \ov T(v_0,v_1,v_2)\big)\to v\ov Rv_2\big].
$$
In other words, $A_0\otimes A_1$ is the set $\{x\in X:\<P,A_0,A_{1}\>\models \ph[x]\}$ of all elements of $X$ that satisfy $\ph$ in $\<P,A_0,A_{1}\>$ when the element is taken as the value of the free variable $v$ of $\ph$.

There are a number of possible properties of the fusion operation on $P^+$ that correspond to a first-order $\sL^*$-condition on $P$, including $\otimes$ being associative, commutative, square-increasing ($A\sub A\otimes A$) and
right-lower-bounded ($A_0\otimes A_1\sub A_1$) \cite[Section 6]{dunn:cano05}. For instance, $\otimes$ is commutative iff $P$ satisfies the sentence
\begin{equation} \label{comfuse}
\forall v_0\forall v_1\forall v_2(\ov T(v_0,v_1,v_2)\leftrightarrow \ov T(v_1,v_0,v_2)).
\end{equation}
Thus the class of all Lambek frames with commutative $\otimes$ is elementary.

There are binary operations $A_0\backslash A_1$ and  $A_0 /A_1$ on $P^+$ that are first-order $\sL^*_2$-definable and are left and right residuals of $\otimes$, meaning they have
$$
B\le A\backslash C  \quad\text{iff}\quad   A\otimes B\le C     \quad\text{iff}\quad A\le C/B.
$$
These residuals are given by
\begin{align*}
A_0\backslash A_1 &=\{ x\in X: \forall x_0\forall y(x_0\in A_0\ \&\ y\in\rho_R A_1\to T(x_0,x,y))\},
\\
A_0/ A_1                 &=\{ x\in X: \forall x_1\forall y(x_1\in A_1\ \&\   y\in \rho_R A_0\to T(x,x_1,y))\},                               
\end{align*}
indicating that they are first-order $\sL^*_2$-definable. We also have
\begin{align*}
 A_0\backslash A_1 &= \bigcap\{T[x_0,-,y]:  x_0\in A_0\ \&\ y\in \rho_R A_1 \},
 \\
 A_0/ A_1                 &=\bigcap\{T[-,x_1,y]:  x_1\in A_1\ \&\   y\in \rho_R A_0 \},
\end{align*} 
showing that $A_0\backslash A_1$ and $A_0/ A_1 $ are stable in a Lambek frame.

Fusion preserves all joins in both coordinates, so is a complete normal operator on $P^+$. $A_0/ A_1$ preserves meets in its \emph{numerator} $A_0$, but turns joins into meets in its \emph{denominator} $A_1$.  $A_0\backslash A_1$ behaves likewise with respect to its numerator $A_1$ and denominator $A_0$ \cite[Chap.~3]{gala:resi07}.
If $A_0/ A_1$ is viewed as a map  from $P^+\times (P^+)^\partial$ to $P^+$, where $(P^+)^\partial$ is the order-dual of $P^+$, then it preserves meets in both arguments, so becomes a complete normal dual operator. Likewise for $A_0\backslash A_1$ as a map $(P^+)^\partial\times P^+\to P^+$.

Put $\Om=\{\otimes,\backslash,/\}$ and take $\Sig$ to be the class of all Lambek frames that satisfy \eqref{comfuse}. If $\Phi$ consists of the  $\sL^*_2$-formulas defining $\otimes$ and its residuals over Lambek frames, then $\Sig_\Phi=\Sig$, an elementary class.  
$\{P^+_\Om: P\in\Sig_\Phi\}$ is a class of commutative residuated lattices.
\qed
\end{example}
 
 \begin{example} \textbf{Modal operators.}
 We follow  \cite{conr:cate16,conr:algo16} in modelling a pair of unary modalities, $\bo$ and $\di$, using polarities with an additional binary relation  $T\sub X\times Y$. This requires that for each $x\in X$ and $y\in Y$, the  $T$-sections 
$$
T[x,-]=\{y\in Y:xTy\} \quad \text{and} \quad T[-,y]=\{x\in X:xTy\}
$$
are stable. The class  of such structures is definable by $\sL^*$-sentences, where  $\sL^*=\sL\cup \{ \ov T\}$ with $\ov T$ a binary relation symbol. For $A\sub X$, define
\begin{alignat*}{3}   
&\bo A &&=\  \{x\in X:\rho_R A\sub T[x,-]\} &&=  \bigcap\{T[-,y]:y\in\rho_R A\},
\\
&\di A &&= \lam\{y\in Y: A\sub T[-,y]\}&&=\bigcap\{\lam\{y\}: A\sub T[-,y]\}.
\end{alignat*}
Then in fact
\begin{align*}
\bo A& = \{x: \forall y[\, \forall z(z\in A\to zRy)\to xTy\,]\}
\\
\di A& = \{x:\forall y[\, \forall z(z\in A\to zTy)\to xRy\,] \},
\end{align*}
indicating that $\bo A$ and $\di A$ are first-order $\sL^*_1$-definable. Both are stable subsets of $X$, being intersections of families of stable sets. $\bo$ and $\di$ are isotone as (first-order definable) operations on $(\wp X,\sub)$, and  $\di$ is left adjoint to $\bo$ in the sense that for any $A,B\sub X$,
$$
\di A\sub B \quad\text{iff}\quad A\sub\bo B.
$$
It is a standard fact that a left adjoint on a complete lattice preserves all joins, while its right adjoint preserves all meets. Thus $\di$ is a complete normal operator on $P^+$, and $\bo$ is a complete normal dual operator.
 \qed
 \end{example}
 
 \section{Ultraproducts of polarity-based structures.}
 
 We recall the definition of  the ultraproduct $\prod_U X_i$ of a collection $\{X_i:i\in I\}$ of sets modulo an ultrafilter $U$ on the index set $I$.  Define an equivalence relation $\sim_U$ on the  direct product $\prod_I X_i$ by putting
 $f\sim_U g$ iff $\{i\in I:f(i)=g(i)\}\in U$, and let $f^U$ be the equivalence class of $f$. Then  
 $\prod_U X_i=\{f^U:f\in \prod_I X_i\}$.
 
 If $\{P_i=(X_i,Y_i,R_i):i\in I\}$ is a set of polarities,
 the ultraproduct $\prod_U P_i$ is defined to be the polarity $(\prod_U X_i,\prod_U Y_i, R^U)$, where the binary relation 
 $R^U$ from  $\prod_U X_i$ to $\prod_U Y_i$ has 
 $$
 f^UR^Ug^U \quad\text{iff} \quad \{i\in I:f(i)R_ig(i)\}\in U.
 $$
 More generally, let each $P_i$ be an $\sL^*$-structure, where $\sL^*$ is an expansion of the signature $\sL$ for polarities by the addition of some finitary relational symbols $\ov T$. Then an $n$-ary $\ov T$ will denote an $n$-ary relation $T_i$ on $X_i\cup Y_i$ for all $i\in I$. We then define $T^U$ on  $\prod_U X_i\cup\prod_U Y_i$ by
$$
T^U (\vc{f^U}{n})\quad\text{iff} \quad \{i\in I:T_i(f_0(i),\dots,f_{n-1}(i))\}\in U.
$$ 
In this way we obtain the  ultraproduct $\prod_U P_i$ as an $\sL^*$-structure.
When all the factors $P_i$ are equal to a single  $P$, then the ultraproduct is the \emph{ultrapower} 
$P^U$ of $P$ modulo $U$.  

 If $\ph(\vc{v}{n})$ is an $\sL^*$-formula and $\vc{f}{n}\in (\prod_I X_i)\cup(\prod_I Y_i)$, let
 \[
 \br{\ph(\vc{f}{n})}=\{i\in I:P_i\models \ph[f_0(i),\dots,f_{n-1}(i)]\}.  
 \]
\L o\'s's Theorem \cite[4.1.9]{chan:mode73} states that  
$$\textstyle
\prod_U P_i\models\ph[f_0^U,\dots,f_{n-1}^U] \quad\text{iff}\quad 
 \br{\ph(f_0,\dots,f_{n-1})}\in U.
 $$
Hence if $\ph$ is a sentence, then  $\prod_U P_i\models\ph$ iff $\{i:P_i\models\ph\}\in U$. This implies that if a class of 
$\sL^*$-structures  is elementary, i.e.\ is the class of all models of some set of $\sL^*$-sentences, then it must be closed under ultraproducts.

 \L o\'s's theorem can be reformulated as a result about definable sets.
If $\al\in\prod_I\wp X_i$ and $f\in \prod_I X_i$, then $\al^U\in\prod_U\wp X_i$ and we define
\begin{equation}  \textstyle  \label{defthet}
\thet(\al^U)=\{f^U\in \prod_U X_i: \{i\in I:f(i)\in\al(i)\}\in U\}.
\end{equation}
 $\thet(\al^U)$ is a well-defined function of $\al^U$, since the righthand set of equation \eqref{defthet} is unchanged if $\al$ is replaced by any $\al'\in\prod_I\wp X_i$ with $\al^U=\al'\,^U$.

\begin{lemma} \label{Los}    \emph{\cite[Lemma 5.2]{gold:cano18}}
Let $\ph(v_0,\dots,v_{n-1})$ be any $\sL^*$-formula. Suppose $\al\in\prod_I\wp X_i$ and $f_0,\dots,f_{n-1} \in (\prod_I X_i)\cup(\prod_I Y_i)$. If for all $i\in I$,
\[
\al(i)=\{x\in X_i:P_i\models\ph[x,f_0(i),\dots,f_{n-1}(i)]\},
\]
then \enspace
$   \textstyle
\thet(\al^U)= \{f^U\in\prod_UX_i : \prod_U P_i\models\ph[f^U,f_0^U,\dots,f_{n-1}^U]\}.
$
\qed
\end{lemma}

Now take a set $\Om$ of operation symbols, a set
 $\Phi=\{\ph_\f:\f\in\Om\}$ of $\sL^*_\om$-formulas having one free variable, and a class $\Sig$ of $\sL^*$-structures. Let
  $\Sig_\Phi$  be the class of all  $P\in\Sig$ such that  $P^+$ is closed under  $f^P_\ph$ for all $\ph\in\Phi$ and hence gives rise to the  $\Om$-lattice $P^+_\Om$ of \eqref{Pptau}.

\begin{theorem}  \label{Fhom}
Let $\Sig_\Phi$ be closed under ultraproducts.
For any  collection $\{P_i:i\in I\}\sub\Sig_\Phi$ and ultrafilter $U$ on $I$, the map $\al^U\mapsto\thet(\al^U)$ is an $\Om$-lattice monomorphism
$$\textstyle\thet\colon \prod_U(P_i)^+_\Om   \mono    ( \prod_UP_i)^+_\Om $$ 
from the $U$-ultraproduct of the stable set $\Om$-lattices $(P_i)^+_\Om$ into the stable set $\Om$-lattice of the ultraproduct $\prod_U P_i$.\end{theorem}

\begin{proof}   In \cite[Theorem 5.3]{gold:cano18} is was shown that $\thet$ is a lattice monomorphism from
  $ \prod_U(P_i^+)$ into    $( \prod_UP_i)^+ $. It suffices then to show that $\thet$ preserves the operations indexed by $\Om$.

Let $\f\in\Om$ be $n$-ary.  In each $\Om$-lattice $(P_i)^+_\Om$, $\f$ is assigned the function $\f^{P_i^+}$ defined by 
$\ph_\f$ (see \eqref{Pptau}). Since $\Sig_\Phi$ is closed under ultraproducts, it contains $\prod_UP_i$, so in  the $\Om$-lattice $ ( \prod_UP_i)^+_\Om $, $\f$ is assigned the function $\f^{( \prod_UP_i)^+} $ defined by $\ph_\f$.

In the ultraproduct of $\Om$-lattices $\prod_U(P_i)^+_\Om$, the definition of $\f^{\prod_U(P_i^+)}$ is that
\begin{equation}  \label{deffbeta}
\f^{\prod_U(P_i^+)}(\al_0^U,\dots,\al_{n-1}^U)=\beta^U, 
\end{equation}
where $\beta\in\prod_I(P_i^+)$ has
$$
\beta(i)=\f^{P_i^+}(\al_0(i),\dots,\al_{n-1}(i))
$$
for all $i\in I$. Thus by \eqref{Fdef} with $P=P_i$,
\begin{equation} \label{beta(i)}
\beta(i)=\{x\in X_i:\<P_i,\al_0(i),\dots,\al_{n-1}(i)\>\models \ph_\f[x]\}.
\end{equation}
To prove the Theorem we need to show that
$\thet$ is a homomorphism for the functions $\f^{\prod_U (P_i^+)}$ and $\f^{( \prod_UP_i)^+}$,  which means that
\begin{equation} \label{homthet}
\thet\big(\f^{\prod_U (P_i^+)} (\al_0^U,\dots,\al_{n-1}^U)\big)=
\f^{( \prod_UP_i)^+}  \big(\thet(\al_0^U),\dots,\thet(\al_{n-1}^U)\big) 
\end{equation}
for all $\al_0,\dots,\al_{n-1}\in\prod_I(P_i^+)$.

Now the $U$-ultraproduct of the $\sL_n^*$-structures $\<P_i,\al_0(i),\dots,\al_{n-1}(i)\>$ is the structure
$\<\prod_UP_i,\thet(\al_0^U),\dots,\thet(\al_{n-1}^U)\>$ in which $S_m$ is interpreted as $\thet(\al_{m}^U)$.
This is because $S_m$ is interpreted as $\al_m(i)$ in each $P_i$, hence is interpreted in $\prod_UP_i$ as 
$$\textstyle
\{f^U\in \prod_U X_i: \{i\in I:f(i)\in\al_m(i)\}\in U\},
$$
which is $\thet(\al_{m}^U)$ by  \eqref{defthet}). 
Thus by Lemma \ref{Los} and \eqref{beta(i)},
$$\textstyle
\thet(\beta^U)=
\{f^U\in\prod_UX_i:\<\prod_UP_i,\thet(\al_0^U),\dots,\thet(\al_{n-1}^U)\>\models\ph_\f[f^U]\}.
$$
This implies, by \eqref{Fdef} with $P=\prod_UP_i$ and $A_m=\thet(\al_m^U)$, that
$$
\thet(\beta^U)=
\f^{(\prod_UP_i)^+}(\thet(\al_0^U),\dots,\thet(\al_{n-1}^U)).
$$
But from \eqref{deffbeta},
$$
\thet(\beta^U)=
\thet(\f^{\prod_U(P_i^+)}(\al_0^U,\dots,\al_{n-1}^U)),
$$
so the last two equations imply the desired equation \eqref{homthet}.
\end{proof}
 The ultrapower case of this theorem  states that if $P\in\Sig_\Phi$, then $\thet$ is an $\Om$-monomorphism 
$$
(P^+_\Om)^U\mono (P^U)^+_\Om.
$$
At the  lattice level $\thet$ is a lattice embedding
$ (P^+)^U\mono (P^U)^+$  that was shown in \cite[Theorem 6.1]{gold:cano18} to give a MacNeille completion of the ultrapower $ (P^+)^U$. We now extend that fact to the $\Om$-lattice level.

\begin{lemma}    \label{completeMac}
Let $\Sig_\Phi$ be closed under ultraproducts.
Let $P\in\Sig_\Phi$,
$U$ be any ultrafilter on a set $I$, and  $\f\in\Om$.
\begin{enumerate}[\rm(1)]
\item 
If $\f^{ (P^U)^+}$ is a complete normal operator, then it is the lower MacNeille extension of $\f^{ (P^+)^U}$.
\item 
If $\f^{ (P^U)^+}$ is a complete normal dual operator,, then it is the upper MacNeille extension of $\f^{ (P^+)^U}$.
\end{enumerate}
\end{lemma}

\begin{proof}
We demonstrate the proof of (1) for the case that $\f$ is binary, since this typifies the general case.
Since $\thet\colon (P^+)^U\mono (P^U)^+$ a MacNeille completion of $ (P^+)^U$, we know that each $A\in  (P^U)^+$ is the join of members of $\image\thet$, hence 
$$
A=\join\{\thet(B):B\in (P^+)^U\ \&\ \thet(B)\le A\}.
$$
 Using this and the hypothesis that  $\f^{ (P^U)^+}$ preserves all joins, we get from \eqref{joincomplete} that for any $A_1,A_2\in  (P^U)^+$, the element $\f^{ (P^U)^+}(A_1,A_2)$ of  $(P^U)^+$  is equal to
$$
\join\{\f^{ (P^U)^+}(\thet(B_1),\thet(B_2)): B_i\in (P^+)^U\ \&\ \thet(B_i)\le A_i \text{ for }i=1,2\}.
$$
But $\f^{ (P^U)^+}(\thet(B_1),\thet(B_2))=\thet(\f^{(P^+)^U}(B_1,B_2))$, since $\thet$ is an $\Om$-homomorph-ism by Theorem \ref{Fhom}. So we get that $\f^{ (P^U)^+}(A_1,A_2)$ is equal to
$$
\join\{\thet(\f^{(P^+)^U}(B_1,B_2)): B_i\in (P^+)^U\ \&\ \thet(B_i)\le A_i \text{ for }i=1,2\}.
$$
But a complete normal operator is isotone, so this last join 
 is  the lower MacNeille extension $ \ov{\f^{ (P^+)^U} }(A_1,A_2) $ by \eqref{lowerMc}.

That proves (1). An order-dual argument gives (2).
\end{proof}

Now suppose $\Om$ is given as a disjoint union $\Lam\cup\Ups$, allowing us to define $\LL^\sg$ and $\ov\LL$ for any $\Om$-lattice $\LL$ according to \eqref{Lsigom} and \eqref{Lovom}. Then we can formulate the following, one of the principal results of this paper.

\begin{theorem}  \label{ephienlarge}
Let $\Sig_\Phi$ be closed under ultraproducts, and suppose that for each $P\in\Sig_\Phi$ and $\f\in\Om$, $\f^{P^+}$ is a complete normal operator if $\f\in\Lam$, and a complete normal dual operator  if $\f\in\Ups$. Then for any $P\in\Sig_\Phi$ there is an ultrafilter $U$ and an $\Om$-monomorphism 
$(P^+_\Om)^\sg\mono (P^U)^+_\Om$
from the canonical extension of the $\Om$-lattice $P^+_\Om$ into the stable set $\Om$-lattice of the  $U$-ultrapower of $P$.
\end{theorem}

\begin{proof}      
Put $ \LL= P^+_\Om=(P^+,\{\f^{P^+}:\f\in\Om\})$ in Theorem \ref{LLsigmaembed}. Since operators and dual operators are isotone, we conclude that there is an $\Om$-monomorphism  $(P^+_\Om)^\sg\mono \ov{(P^+_\Om)^U}$ for some $U$. Then it is enough to show that we can take the MacNeille completion $\ov{(P^+_\Om)^U}$ to be $(P^U)^+_\Om$. Now
\begin{equation*}  \label{PplustauU}
(P^+_\Om)^U=( (P^+)^U, \{\f^{(P^+)^U}:\f\in\Om \} ),
\end{equation*}
so by \eqref{Lovom},
$$
\ov{(P^+_\Om)^U}=\Big( \ov{(P^+)^U}, \Big\{\ov{\f^{(P^+)^U}}:\f\in\Lam \Big\}, \Big\{\widehat{\f^{(P^+)^U}}:\f\in\Ups\Big\} \Big).
$$
By  \cite[Theorem 6.1]{gold:cano18} we can take $\ov{(P^+)^U}$ to be $(P^U)^+$. Then  Lemma \ref{completeMac} and the hypotheses of this Theorem give that
$\ov{\f^{(P^+)^U}}= \f^{ (P^U)^+}$ when $\f\in\Lam$,
and
$\widehat{\f^{(P^+)^U}}= \f^{ (P^U)^+}$ when $\f\in\Ups$. 
Hence
$$
\ov{(P^+_\Om)^U}=((P^U)^+,\{\f^{(P^U)^+}:\f\in\Om\}) = (P^U)^+_\Om,
$$
giving the desired conclusion.
\end{proof}

\section{Generating varieties  closed under canonical extensions}

We now introduce the notion of a canonicity framework and put together our results so far in order to derive our main goal.

The use of the symbols $\mono$ and $\epi$ will be extended to have them denote binary relations between algebras, writing 
$\A\mono\B$ to mean that \emph{there exists} an injective homomorphism from $\A$ to $\B$, and $\A\epi\B$ to mean that there exists an surjective one. 
We consider a situation involving the following four ingredients:
\begin{itemize}
\item 
A class $\Sig$ of structures, of some type, that is closed under ultraproducts.
\item
A variety $\CC$ of algebras of some given algebraic signature.
\item
An operation $(-)^\sg\colon\CC\to\CC$ assigning to each algebra $\A\in\CC$ another algebra $\A^\sg\in\CC$.
\item
An operation $(-)^+\colon\Sig\to\CC$ assigning to each structure $P\in\Sig$ an algebra $P^+\in\CC$.
\end{itemize}
The list $\<\Sig,\CC,(-)^\sg,(-)^+\>$ of these ingredients is  called a \emph{canonicity framework} if it satisfies the following axioms
for all $\A,\B\in\CC$, all indexed subsets $\{P_i:i\in I\}$ of $\Sig$,  and all $P\in\Sig$.

\begin{enumerate}[({A}1)]
\item 
If  $\A\mono\B$ then  $\A^\sg\mono\B^\sg$, and if $\A\epi\B$ then $\A^\sg\epi\B^\sg$.
\item
$ \prod_U(P_i^+)   \mono    ( \prod_UP_i)^+ $, for any ultrafilter $U$ on $I$.
\item
There exists an ultrafilter $U$ such that $(P^+)^\sg\mono (P^U)^+$.
\item
$\big(\prod_I (P_i^+)\big)^\sg \mono \underset{U\in \beta I}{\prod} \big(\prod_U (P_i^+)\big)^\sg$, where $ \beta I$ is the set of all ultrafilters on I.
\end{enumerate}
It was shown in \cite[Theorem 7.1]{gold:cano18} that these axioms yield the following result:
\begin{quote}\em
In any canonicity framework, if $\cS$ is any subclass of $\Sig$ that is closed under ultraproducts, then
the variety of algebras generated by  $\cS^+=\{P^+:P\in\cS\}$ is closed under the  operation     $(-)^\sg$.
\end{quote}
Now given a class of the form $\Sig_\Phi$ (for some $\Om=\Lam\cup\Ups$) that satisfies the description in the first sentence of  Theorem \ref{ephienlarge}, we can  construct a canonicity framework  by taking 
\begin{itemize}
\item
$\Sig$ to be $\Sig_\Phi$;
\item
$\CC$ to be the variety of all $\Om$-lattices $\LL$  in which each $\f^\LL$  is a normal operator if $\f\in\Lam$, and a normal dual operator  if $\f\in\Ups$;
\item
$(-)^\sg$ to be the operation $\LL\mapsto\LL^\sg$ defined in \eqref{Lsigom};
\item
$(-)^+$ to be the operation $P\mapsto P^+_\Om$ as defined in \eqref{Pptau}.   
\end{itemize}
We verify that these definitions fulfil the canonicity framework axioms:
\begin{enumerate}
\item[(A1):]
$\CC$ is a variety  of $\Om$-lattices whose members are \emph{monotone}, since operators and dual operators are isotone in each variable. It was shown in \cite[Theorem 5.4]{gehr:boun01} that the operation $f\mapsto f\lo$ of  lower canonical extension of maps between monotone lattice expansions preserves homomorphisms, injectivity and surjectivity. So for $\LL,\M\in\CC$,  the existence of an  injective or surjective $\Om$-homomorphism $f\colon\LL\to\M$ guarantees the existence of an $\Om$-homomorphism $f\lo:\LL^\sg\to\M^\sg$ with the same property, giving (A1).
Note that $\CC$ is closed under $(-)^\sg$ because if $f$ is an operator then so is $f\lo$,  and if $f$ is a dual operator then so is $f\up$ \cite[Lemma 4.6]{gehr:boun01}. Also  $\{P^+_\Om:P\in\Sig_\Phi\}\sub\CC$ by definition of $\Sig_\Phi$.

\item[(A2):] 
Theorem \ref{Fhom} shows that $\textstyle \prod_U(P_i)^+_\Om   \mono    ( \prod_UP_i)^+_\Om $.

\item[(A3):]  Theorem \ref{ephienlarge} provides any $P\in\Sig_\Phi$ with a $U$ such that
                      $(P^+_\Om)^\sg\mono (P^U)^+_\Om$.

\item[(A4):]  This is an instance of the stronger fact that for any  set $\{\LL_i:i\in I\}$ of $\Om$-lattices  there is an isomorphism
between
$\big(\prod_I\LL_i\big)^\sg$ and  $\prod_{U\in \beta I}\big(\prod_U\LL_i\big)^\sg$
\cite[Theorem 3.1]{gold:cano18}.
\end{enumerate}

Since (A1)--(A4) are satisfied, Theorem 7.1 of
\cite{gold:cano18} (as quoted above) delivers our goal of generalising Fine's canonicity theorem to lattice-based algebras:

\begin{theorem}
Suppose that  $\Sig_\Phi$ is closed under ultraproducts, and  the symbols of $\Lam$ denote complete normal operators in each member of $\Sig_\Phi$, while the symbols of $\Ups$ denote complete normal dual operators.
If $\cS$ is any subclass of\/ $\Sig_\Phi$ that is closed under ultraproducts, then 
the variety of lattice expansions generated by  $\{P^+_\Om:P\in\cS\}$   is closed under canonical extensions.
\qed
\end{theorem}

In conclusion we observe that the definition of a canonicity framework emphasises the algebraic side of the duality between algebras $\A\in\CC$ and structures $P\in\Sig$. The axioms (A1)--(A4) all describe relationships between algebras, with the role of the ultraproducts-closed class $\Sig$ being largely to supply some of the algebras via the map 
$(-)^+\colon\Sig\to\CC$. What is missing is a map $(-)_+\colon\CC\to\Sig$ in the reverse direction, assigning to each algebra $\A$ a \emph{canonical structure} $\A_+\in\Sig$. Moreover, while  $\CC$  forms a category under the standard notion of homomorphism between algebras, we are missing a suitable notion of ``morphism'' between the structures in $\Sig$.
These desiderata are present in the Boolean case, where we have  $(\A_+)^+=\A^\sg$ and can show that for any structure $P$ there is an ultrapower $P^U$ that can be mapped by a \emph{bounded morphism} (a.k.a.\ p-morphism) onto 
$(P^+)_+$ \cite[Theorem 3.6.1]{gold:vari89}. This bounded morphism induces an embedding  in the reverse direction from
$((P^+)_+)^+$, which is  $(P^+)^\sg$, into $(P^U)^+$, thereby proving (A3).

Ultimately what we want in the non-distributive setting is to make $\Sig_\Phi$ into a category whose duality with $\CC$ is expressed by the existence of a pair of contravariant functors between them. That would allow a version of the Goldblatt-Thomason theorem to be formulated, giving structural conditions on a subclass of $\Sig_\Phi$ that characterise when that subclass is equal to $\{P:P^+\in\V\}$ for some subvariety $\V$ of $\CC$.
%This will be the subject of further work.
A functorial duality of this kind is developed in \cite{gold:morp19}, where it is shown that it provides two non-equivalent versions of the Goldblatt-Thomason theorem.

%%%%%%  THE BIBLIOGRAPHY

\bibliographystyle{plain}
\small

%\bibliography{goldblatt}

%%%%%%%%%%%%%%%%%%%%%%%%%%%%%%%%%%%%%%%%%%%%%%%%%%%%%%%%%%%%%%
\end{document}